\documentclass[11pt]{article}
\usepackage{amssymb,amsmath,amsthm,amsfonts,verbatim}
\usepackage[all,cmtip]{xy}
\input diagxy
\usepackage{hyperref}
\usepackage[top=1.2in,bottom=1.2in,left=1.21in,right=1.21in]{geometry}
\usepackage{todonotes}
\usepackage{qtree}

\newtheorem{theorem}{Theorem}[section]

\newtheorem{corollary}[theorem]{Corollary}
\newtheorem{lemma}[theorem]{Lemma}

\newtheorem{definition}[theorem]{Definition}

\theoremstyle{definition}

\newtheorem{example}[theorem]{Example}

\numberwithin{equation}{section}

\theoremstyle{remark}
\newtheorem{remark}[theorem]{Remark}

\newcommand{\into}{\hookrightarrow}

\newcommand{\iso}{\cong}
\newcommand{\op}{op}

\newcommand{\Ab}{\mathbb{A}}

\newcommand{\Cb}{\mathbb{C}}
\newcommand{\C}{\mathbb{C}}
\newcommand{\Eb}{\mathbb{E}}
\newcommand{\Fb}{\mathbb{F}}
\newcommand{\F}{\mathbb{F}}

\newcommand{\Nb}{\mathbb{N}}

\newcommand{\Qb}{\mathbb{Q}}
\newcommand{\Q}{\mathbb{Q}}

\newcommand{\Zb}{\mathbb{Z}}
\newcommand{\Z}{\mathbb{Z}}
\newcommand{\V}{\mathcal V}
\newcommand{\Kbar}{\overline{K}}
\newcommand{\Fqbar}{\overline{\Fb}_q}
\newcommand{\Fpbar}{\overline{\F}_p}
\newcommand{\et}{{et}}
\newcommand{\tr}{Tr}

\newcommand{\Schemes}{\mathbf{Schemes}}

\newcommand{\Spec}{Spec}

\DeclareMathOperator{\colim}{colim}

\DeclareMathOperator{\Sym}{Sym}

\DeclareMathOperator{\PConf}{PConf}
\DeclareMathOperator{\UConf}{UConf}

\DeclareMathOperator{\Frob}{Frob}
\DeclareMathOperator{\Ind}{Ind}
\DeclareMathOperator{\Hom}{Hom}
\DeclareMathOperator{\Fix}{Fix}
\DeclareMathOperator{\Gal}{Gal}
\DeclareMathOperator{\diag}{diag}
\DeclareMathOperator{\res}{res}
\DeclareMathOperator{\Emb}{Emb}

\title{\'Etale homological stability and arithmetic statistics}

\author{Benson Farb and Jesse Wolfson \thanks{B.F. is supported in part by NSF Grant Nos. DMS-1105643 and DMS-1406209. J.W. is supported in part by NSF Grant No. DMS-1400349.}}

\begin{document}
\maketitle
\begin{abstract}
We contribute to the arithmetic/topology dictionary by relating asymptotic point counts and arithmetic statistics over finite fields to homological stability and representation stability over $\Cb$ in the example of configuration spaces of $n$ points in smooth varieties. To do this, we  import the method of homological stability from the realm of topology into the theory of \'{e}tale cohomology; in particular we give the first examples of stability of \'{e}tale cohomology groups as Galois representations where the Galois actions are not already explicitly known. We then establish subexponential bounds on the growth of the unstable cohomology, and we apply this and \'etale homological stability to compute the large $n$ limits of various arithmetic statistics of configuration spaces of varieties over $\F_q$.
\end{abstract}

\section{Introduction}
Let $X$ be a scheme defined over $\Z$.  The Weil conjectures provide a fundamental link between the topology of $X(\C)$ and the arithmetic of $X(\F_q)$.  As first indicated by work of Ellenberg-Venkatesh-Westerland \cite{EVW} and then by Church-Ellenberg-Farb \cite{CEF2}, stability phenomena in topology should have, via this remarkable correspondence, implications on the arithmetic side.  We summarize this in the following table, with the rows going from least to most general.



\medskip
\begin{center}
\begin{tabular}{|c|c|}
\hline
{Topology} \qquad\quad& {Arithmetic} \\
\hline
$H^*(X(\C))$& $|X(\F_q)|$\\
\hline
homological stability of $X_n$ & asymptotics of $|X_n(\F_q)|$ as $n\rightarrow\infty$\\
\hline
representation stability & asymptotics of arithmetic \\
& statistics on $X_n(\F_q)$\\
\hline
\end{tabular}
\end{center}

One of the main goals of the present paper is to realize the bottom rows of this table for varieties $X_n$ of configurations of $n$ points (ordered and unordered) on a given smooth variety.  We view our results as giving a ``proof of concept'' for these  additions to the topology/arithmetic dictionary.  We accomplish this in two steps:
\begin{enumerate}
\item We prove what we call {\em \'{e}tale homological (and representation) stability} for $H^*_{\et}(X_{n/\Fqbar};\Q_\ell)$.  This allows us to break up
$H^*_{\et}(X_{n/\Fqbar};\Q_\ell)$ into tow parts: stable and unstable.
\item We obtain sub-exponential bounds on the growth of the unstable part of $H^*_{\et}(X_{n/\Fqbar};\Q_\ell)$.  This allows us to prove that this unstable part does not contribute to the limiting density of $|X_n(\F_q)|$.  As we explain below, the absence of such bounds is a significant obstruction to understanding the asymptotic point counts of many families of interest.
\end{enumerate}

In the \'{e}tale context, stability of each $H^i _{\et}(X_{n/\Fqbar};\Q_\ell)$ as a Galois representation, not just as a vector space, is crucial.  It is the difference between proving that limits such as $\lim_{n\rightarrow\infty}q^{-\dim X_n}|X_n(\F_q)|$ exist, and actually computing the limiting answer.

\paragraph{Homological stability. } A sequence $\{X_n\}$ of spaces or groups is said to satisfy {\em homological stability} over a ring $R$ if $H_i(X_n;R)$ or $H^i(X_n;R)$ is independent of $n$ for $n\geq D(i)$; the number $D(i)$ is called the {\em stable range}.   Typically, but not always, there are maps $\psi_n:X_n\to X_{n+1}$ or $\phi_n: X_{n+1}\to X_n$ inducing isomorphisms
\[ (\psi_n)_*:H_i(X_n;R)\to H_i(X_{n+1};R)\ \ \ \text{or}\ \ \ \phi_n^*:H^i(X_n;R)\to H^i(X_{n+1};R).\]  Examples of $X_n$ satisfying homological stability include classifying spaces of symmetric groups $S_n$ (Nakaoka), arithmetic groups like ${\rm SL}_n\Z$ (Borel), the moduli spaces ${\cal A}_g$ (Borel) and ${\cal M}_g$ (Harer), and also configuration spaces $\UConf_n(M)$ of  unordered $n$-tuples of distinct points on a manifold $M$ (Arnol'd, McDuff, Segal, Church).  Homological stability has been a powerful tool in topology.  It converts an {\it a priori} infinite computation to a finite one.  Further, the stable answer, $H_i(X_n;R)$ for $n\geq D(i)$, can often be computed explicitly.

Many natural sequences $X_n$ come equipped with actions of groups $G_n$ by automorphisms.
A basic example is the space $\PConf_n(M)$ of ordered $n$-tuples of distinct points on a manifold $M$, on which the symmetric group $S_n$ acts by permuting the ordering.  Such spaces almost never satisfy homological stability, but they instead often satisfy {\em representation stability} : the decomposition of $H^i(X_n;\Q)$ into a sum of irreducible $S_n$-representations stabilizes in a precise sense (see \cite{CF,CEF1}, \S\ref{sec:repstabsum} below, and \cite{Fa} for a survey).   When $R=\Q$, plugging the trivial representation into this theory gives classical homological stability for the sequence $X_n/S_n$.    So for example representation stability for $\PConf_n(M)$ gives classical homological stability for
the space $\UConf_n(M)=\PConf_n(M)/S_n$ of unordered $n$-tuples of distinct points on $M$; see \cite{Ch}.  The theory of representation stability, initiated by Church, Ellenberg and Farb, is currently undergoing a rapid development.

\paragraph{\'{E}tale homological stability.} Consider a scheme $Y$, smooth over $\Z[1/N]$ for some $N$.  We can extend scalars to $\C$ and consider the complex points $Y(\C)$, and we can also reduce modulo $p$ for any prime $p\nmid N$.  This gives a variety defined over $\F_p$, and for any positive power $q=p^d$ we can consider both the $\F_q$-points as well as the $\Fqbar$-points of $Y$, where $\Fqbar$ is the algebraic closure of $\F_q$.

One of the most fundamental arithmetic invariants attached to $Y$ is its
{\em \'{e}tale cohomology} $H^*_{\et}(Y_{/\Kbar};\Q_\ell)$, where $K$ is a number field or finite field of characteristic prime to $N$, and where $\ell\neq p$ is prime.  The Galois action on $Y_{/\Kbar}$ induces a Galois action on each $\Q_\ell$-vector space $H^i_{\et}(Y_{/\Kbar};\Q_\ell)$, and this action is a crucial part of the data.

Now let $X_n$ be a sequence of schemes that are smooth over $\Z[1/N]$ for some $N$, for example $X_n=\PConf_n(Y)$ or $X_n=\UConf_n(Y)$ for $Y$ smooth over $\Z[1/N]$.    Given the usefulness of homological stability in topology, one wants to prove such stability for $H^i_{\et}(X_{n/\Kbar};\Q_\ell)$ for $K$ a number field or a finite field with characteristic $p\nmid  N$.   There are a number of different notions of what ``stability'' means in this context (see \S \ref{sec:repstability} below). We adopt the strongest of these possibilities.

\begin{definition}[{\bf \'{E}tale homological stability}]
We say that a sequence $X_n$ of schemes satisfies {\em \'{e}tale homological stability} over
a field $K$ if for each $i\geq 0$ there exists $D=D(i)$ so that the isomorphism type of $H^i_{\et}(X_{n/\Kbar};\Q_\ell)$ as a $\Gal(\Kbar/K)$-representation does not depend on $n$ for $n\geq D$.  The function $D(i)$ is called the {\em stable range}.
\end{definition}

When each $X_n$ in addition admits an $S_n$-action, such as $X_n=\PConf_n(Y)$ or $Y^n$, we have a corresponding notion of {\em \'{e}tale representation stability} over $K$.  This definition is a bit more involved; see \S\ref{subsection:repstab}.  It implies \'{e}tale homological stability for the sequence of varieties $X_n/S_n$.

As far as we know, there have previously been only two results in this direction.  The first is a beautiful paper of Ellenberg-Venkatesh-Westerland \cite{EVW}, where a sequence $X_n$ of certain Hurwitz varieties is considered.  They prove that each $H^i_{\et}(X_{n/\Fqbar};\Q_\ell)$ stabilizes as a $\Q_\ell$ vector space for $n\geq D(i)$, and apply this to prove a function field version of the Cohen-Lenstra Heuristics in number theory.  However, as discussed  explicitly at the end of \S 7 of \cite{EVW}, the argument there does not prove that the $\Gal(\Fqbar/\F_q)$-representations stabilize.

The other previous result in the direction of \'{e}tale homological stability was proved by Church-Ellenberg-Farb \cite{CEF2}.  They showed that each sequence $H^i_{\et}(\PConf_n(\Ab^1)_{/\Fqbar};\Q_\ell)$ is \'{e}tale representation stable, from which it follows that $H^i_{\et}(\UConf_n(\Ab^1)_{/\Fqbar};\Q_\ell)$ satisfies \'{e}tale homological stability.  However, in this case the stability of the $\Frob_q$-action is just a consequence of the fact that this action is semisimple with the eigenvalues known quite explicitly, namely they are all $q^{i}$ on $H^i_{\et}(\PConf_n(\Ab^1)_{/\Fqbar};\Q_\ell)$.

\subsection{Statement of results}
In this paper we give the first examples of \'{e}tale homological and representation stability where the Galois actions are not known explicitly, but are only proved to stabilize. To state our results we need two different descriptions of representations of the symmetric groups $S_n$.

Let $X_i$ be the class function on all symmetric groups $S_n,n\geq 1$ given by setting $X_i(\sigma)$ to be the number of $i$-cycles in the cycle decomposition of $\sigma$.  A {\em character polynomial} is any polynomial $P\in\Q[X_1,X_2,\ldots]$; it is a class function on each $S_n, n\geq 1$.  The {\em degree} of a character polynomial is defined by setting $\deg X_i:=i$.  See \S\ref{sec:repstabsum} below for more details.  As shown by Church-Ellenberg-Farb \cite{CEF1}, character polynomials give a compact and uniform way of describing the character of certain infinite sequences of $S_n$-representations for $n=1,2,\ldots $.

A \emph{partition} of
$n$ is a sequence $\lambda=(\lambda_1\geq \cdots \geq \lambda_r\geq0)$ with $\sum_i\lambda_i=n$.  The irreducible representations $V(\lambda)$ of $S_n$ are classified by partitions $\lambda\vdash n$.  A partition $\lambda\vdash k$ gives a sequence $V(\lambda)_n$ of irreducible $S_n$-representations for $n\geq k+\lambda_1$ by defining
$V(\lambda)_n$  to be the irreducible representation of $S_n$ corresponding to the partition
$(n-k,\lambda_1,\ldots ,\lambda_r)$.  Every irreducible representation of
$S_n$ is of the form $V(\lambda)_n$ for a unique partition
$\lambda$; for example the trivial and standard representations of $S_n$ are $V(0)_n$ and $V(1)_n$, respectively.

\bigskip
\noindent
{\bf Theorem A (\'{E}tale representation stability):} {\it Let $Y$ be a scheme, smooth over $\Z[1/N]$ for some $N$, with geometrically connected fibers.  Suppose that $Y$ is projective over $\Z[1/N]$, or more generally normally compactifiable (see \S\ref{subsection:basechange} below for the precise definition). Let $K$ be either a number field or a finite field of characteristic $p\nmid N$.

For each $i\geq 0$, the sequence $H_{\et}^i({\PConf_n(Y)}_{/\Kbar};\Q_\ell)$ of $\Gal(\Kbar/K)$-modules  satisfies \'{e}tale representation stability  \footnote{See \S\ref{subsection:repstab} for the precise definition of  \'{e}tale representation stability.} over $K$ with stable range
$D(i)=2i$ for $\dim Y\geq 2$ and $D(i)=4i$ for  $\dim Y=1$.  In particular:

\bigskip
\noindent
{\bf 1. Inductive description: } For all $n\geq 0$, there is an isomorphism of $\Gal(\Kbar/K)$-representations:
\begin{equation}
\label{eq:colim}
H^i_{\et}(\PConf_n(Y)_{/\Kbar};\Q_\ell)\cong\ \colim_{S} H^i_{\et}(\PConf_{|S|}(Y)_{/\Kbar};\Q_\ell)
\end{equation}
where the colimit is taken over the poset of all subsets $S\subset \{1,\ldots ,n\}$ such that
$|S|\leq D(i)$.  This gives, for each $n\geq D(i)$, a recipe for building the $\Gal(\Kbar/K)$-representation $H^i_{\et}(\PConf_n(Y)_{/\Kbar};\Q_\ell)$ from a fixed finite collection of $\Gal(\Kbar/K)$-representations.

\bigskip
\noindent
{\bf 2. Stability of isotypics: }
For each character polynomial $P$, there exists a unique virtual  $\Gal(\Kbar/K)$-representation $H^i_{\et}(\PConf(Y))_P$ over $\Q_\ell$, linear in $P$, so that when $P=\chi_{V(\lambda)}$ for some $\lambda\vdash k$, there exists $D$ such that for all $n\ge D$:
\begin{equation*}
    H^i_{\et}(\PConf(Y))_{\chi_{V(\lambda)}}=H^i_{\et}(\PConf_n(Y)_{/\Kbar};\Q_\ell)\otimes_{\Q_\ell[S_n]}V(\lambda)_n
\end{equation*}
and the right hand side is independent of $n$ \emph{as a $\Gal(\Kbar/K)$-representation}.

\bigskip
\noindent
{\bf 3. Polynomial characters: } There exists a character polynomial $Q(X_1,\ldots ,X_r)$ so that for all $n\geq D(i)$:
\[\chi_{H_{\et}^i({\PConf_n(Y)}_{/\Kbar};\Q_\ell)}(\sigma)=Q(X_1(\sigma),\ldots ,X_r(\sigma)) \ \ \ \text{for all $\sigma\in S_n$}.\]
where $\deg(Q)\leq i$ if $\dim Y>1$ and $\deg(Q)\leq 2i$ if $\dim Y=1$.
}

\bigskip
\noindent
{\bf Remarks: }
\begin{enumerate}
\item Our proof of Theorem A shows that Item (1) holds with $\Q_\ell$ replaced by $\Z_\ell$ or $\Z/\ell^n\Z$.
\item In \S\ref{subsection:repstab} we prove the analogous theorem with $\PConf_n(Y)$ replaced by $Y^n$.
\item Nir Gadish \cite{Ga} has recently isolated a concept of finitely generated $I$-poset, for a wide class of categories $I$, and has used this to prove \'etale representation stability for a rich class of sequences of complements of linear subspace arrangements.
\end{enumerate}

Plugging in $P=1$ into Item (2) of Theorem A gives the following.

\bigskip
\noindent
{\bf Corollary A' (\'{E}tale homological stability):} {\it With terminology as in Theorem A, the sequence $H_{\et}^i(\UConf_n(Y)_{/\Kbar};\Q_\ell)$ satisfies \'{e}tale homological stability over $K$: these $\Gal(\Kbar/K)$-representations do not depend on $n$ for $n\geq D(i)$.}

\bigskip
\noindent
{\bf Remarks: }
\begin{enumerate}
\item  \'{E}tale homological stability for $\Sym^nY$ follows from K\"unneth and transfer, just as in the topological case.

\item The stability of \'{e}tale homotopy groups for $\Sym^nY$ was recently proved by Tripathy \cite{Tr}.

\item Quoc Ho  \cite{Ho} has recently given an independent proof of
 Corollary A'  for $Y$ smooth not necessarily normally compactifiable.  His method is based on factorization homology, and is quite different from the methods of this paper.
\end{enumerate}

\paragraph{Stability of arithmetic statistics.}  The application of homological stability to arithmetic statistics was pioneered by Ellenberg-Venkatesh-Westerland in \cite{EVW}.  The  fundamental link is provided by the {\em Grothendieck-Lefschetz Trace Formula} \footnote{Here we have assumed that $Z$ is smooth and applied Poincar\'{e} Duality to the usual Grothendieck-Lefschetz Formula.}:

\begin{equation}
\left|Z(\F_q)\right|= q^{\dim(Z)}\sum_{i\geq 0}(-1)^i
\tr\big(\Frob_q:H^i_{\et}(Z_{/\Fqbar};\Q_\ell)^{*}\to H^i_{\et}(Z_{/\Fqbar};\Q_\ell)^{*}\big)
\end{equation}
and its twisted version (see \eqref{eq:GLtwisted} below).  Given Deligne's theorem
\cite[Theorem 1.6]{De1} that any eigenvalue $\lambda$ of $\Frob_q$ on $H^i_{\et}(Z_{/\Fqbar};\Q_\ell)^*$ satisfies $|\lambda|\leq q^{-i/2}$ , one can bound the number $|Z(\F_q)|$ of $\F_q$-points via
\[|Z(\F_q)|\leq q^{\dim Z} \sum_{i=0}^{2\dim Z}b_iq^{-i/2}\]
where $b_i:=\dim H^i_{\et}(Z_{/\Fqbar};\Q_\ell)^*$.  Applying this reasoning to a sequence $Z_n$ of smooth varieties gives
\begin{equation}
\label{eq:pointcount3}
q^{-\dim Z_n}|Z_n(\F_q)|\leq  \sum_{i=0}^{2\dim Z_n}b_i(n)q^{-i/2}
\end{equation}
where we have emphasized via notation that $b_i$ is a function of $n$.  It seems that \'{e}tale homological stability, namely the fact that $b_i(n)$ is constant for $n\geq D(i)$, should imply that the limit as $n\to\infty$ of the left-hand side of \eqref{eq:pointcount3}  exists. However, it could be that $\dim(Z_n)$ goes to $\infty$ with $n$ and that $b_i(n)$ grows more quickly than $q^{i/2}$, even for any $q$;  this would imply the divergence of the right-hand side of \eqref{eq:pointcount3}.  This super-exponential growth is known to occur in natural examples, for example for $Z_n$ the moduli space of genus $n$ smooth algebraic curves, and also for $Z_n$ the moduli space of $n$-dimensional principally polarized abelian varieties.  In the latter example, recent work of Lipnowski-Tsimerman \cite{LT} shows that this growth actually does change the point count $|Z_n(\F_q)|$, as they show this number grows more quickly than the expected $q^{\dim Z_n}$.

Thus, in order to apply \'{e}tale homological stability to obtain the existence of asymptotic point counts in a given example, it is necessary to prove sub-exponential (in $i$) bounds on $b_i(n)$, independent of $n$.  In other words, control of the {\em unstable} \'{e}tale cohomology $H^i_{\et}({Z_n}_{/\Fqbar};\Q_\ell)^*$ is needed.

Proving such bounds is a major obstruction for arithmetic applications; see \S\ref{section:convergent:cohomology} for a discussion.  This problem is a very special case (namely the case $P\equiv 1$) of more general arithmetic statistics, where one needs a twisted version of the Grothendieck-Lefschetz formula, and where the control on the unstable cohomology is even more difficult to prove; see \S\ref{section:convergent:cohomology} below. A significant part of this paper, \S\ref{section:convergent:cohomology}, is devoted to overcoming this problem for the examples $X^n$ and $\PConf_n(X)$. We obtain the following results.

\bigskip
\noindent
{\bf Theorem B (Bounding the unstable cohomology)}:
{\it Let $X$ be a smooth, orientable manifold with $\dim(H^*(X;\Q))<\infty$ (e.g. $X$ compact).  Then for any character polynomial $P$ there exists a function $F_P(i)$, subexponential in $i$, such that for all $n\geq 1$:
            \[\langle P,(H^i(\PConf_n(X);\Q)\rangle_{S_n}\leq F_P(i).\]
}

In \S\ref{section:arithmetic:stats}, we apply Theorems A and B to obtain the following.

\bigskip
\noindent
{\bf Theorem C (Arithmetic statistics of configuration spaces}): {\it
Let $X$ be a scheme, smooth over $\Z[1/N]$ for some $N$, with geometrically connected fibers.  Suppose that $X$ is projective over $\Z[1/N]$ or, more generally, normally compactifiable at $p\nmid N$.  Let $P$ be any character polynomial, and denote by $H^i_{\et}(\PConf(X))_P$ the stable $P$-isotypic part of the \'etale cohomology of $\PConf_\bullet(X)_{\Fqbar}$, and by $H^i_{\et}(\UConf(X))$ the stable cohomology of $\UConf_\bullet(X)_{\Fqbar}$ (see \S\ref{section:arithmetic:stats} for precise definitions).  Then
\[\lim_{n \to \infty} q^{-n\dim X} \sum_{y \in \UConf_n(X)(\F_q)} P(y)
=\sum_{i=0}^\infty(-1)^i
\tr\big(\Frob_q\circlearrowleft H^i_{\et}(\PConf(X))_P^{*}\big),
\] in particular, both sides of the above converge.  Specializing to $P=1$, we obtain
\begin{equation*}
    \lim_{n\to\infty}q^{-n\dim X} |\UConf_n(X)(\F_q)|=\sum_{i=0}^\infty(-1)^i
\tr\big(\Frob_q\circlearrowleft H^i_{\et}(\UConf(X))^{*}\big).
\end{equation*}
}

\bigskip
A different description of the left hand side of Theorem B, established using analytic methods, will appear in forthcoming work of Weiyan Chen \cite{Che}. En route to proving Theorems B and C, we also prove the analogous statements for $\Sym^n(X)$; see \S\ref{section:arithmetic:stats}.

\paragraph{Acknowledgements. } It is a pleasure to thank  George Andrews, Kathrin Bringmann, Mark Kisin, and Ken Ono for useful discussions. We thank Nir Gadish and Brian Conrad for very helpful conversations about twisted $\ell$-adic sheaves and transfer.  We thank Weiyan Chen, Jordan Ellenberg and Burt Totaro for numerous helpful comments on an earlier draft.  Finally, it is a pleasure to thank Matt Emerton for his careful reading and many detailed comments on an earlier draft of this paper.

\section{\'{E}tale Representation stability}
 \label{sec:repstability}

In this section, we briefly summarize the theory of representation stability and FI-modules, as it is used in topology, as well as some of its consequences.  This theory was developed by Church, Ellenberg and Farb \cite{CF,CEF1}, and later with Nagpal \cite{CEFN}; see \cite{Fa} for a survey.  We refer the reader to these references for details.  We then give a general setup for proving similar stability theorems in \'{e}tale cohomology.

\subsection{Quick summary of representation stability and FI-modules}
\label{sec:repstabsum}

An {\em FI-module} $V$ over a Noetherian ring $R$ is a functor from the category FI of finite sets and injections to the category of $R$-modules.  Thus to each natural number $n$ we have associated an $R$-module $V_n$ with an $S_n$ action, with a map $V_m\to V_n$ for each injection $\{1,\ldots ,m\}\to \{1,\ldots ,n\}$.  Recall that the opposite category
${\rm FI}^{\rm op}$ is the same as FI but with arrows reversed.  A {\em co-FI module} over $R$ is a functor from ${\rm FI}^{\rm op}$ to $R$-modules.  We also have the associated notions of
{\em FI-space, FI-scheme}, etc., and the associated co-FI versions.

An FI-module $V$ is {\em finitely generated} if there is a finite set $S$ of elements  in $\coprod_iV_i$ so that no proper sub-FI-module of $V$ contains $S$.  One of the reasons that we care about finitely-generated FI-modules is the following theorem.

\begin{theorem}[{\bf Structural properties of finitely-generated FI-modules}]
\label{theorem:FIproperties}
Let $V$ be an FI-module over a commutative Noetherian ring $R$.  If $V$ is
finitely-generated then:

\medskip
\noindent
{\bf Representation stability (\cite{CEF1}): } When $R$ is a field of characteristic $0$, finite generation of $V$ implies  representation stability in the sense of \cite{CF} for the sequence $\{V_n\}$ of $S_n$-representations.

\smallskip
\noindent
{\bf Inductive description (\cite{CEFN}): } Let $V_\bullet$ be a finitely-generated FI-module over a Noetherian ring $R$. Then there exists some $N \ge 0$ such that for all $n\in\Nb$, there is a natural isomorphism
\begin{equation*}
    V_n\cong\colim_{S\subset[n],|S|\le N} V_S,
\end{equation*}
i.e. these isomorphisms commute with homomorphisms of $FI$-modules. By definition, the \emph{stable range of $V_\bullet$} $N(V)$ is the minimal such $N$.

\smallskip
\noindent
{\bf Isomorphism of trivial isotypics (\cite{Ch}):} Let $V_\bullet$ be a finitely-generated FI-module over a Noetherian ring $R$ with stable range $N(V)$.  Then for all $n\ge N(V)$ the map $V_n\to V_{n+1}$, given by averaging the structure maps, induces an isomorphism
\begin{equation*}
    V_n^{S_n}\to^\cong  V_{n+1}^{S_{n+1}}.
\end{equation*}
\end{theorem}
We remark that the isomorphism of trivial isotypics illustrates one of the key advantages of considering (co-)FI-spaces: while stabilization maps for many natural sequences of spaces or schemes do not naively exist, a (co-)FI-space $Z_\bullet$ comes equipped with canonical rational correspondences from $Z_{n+1}/S_{n+1}$ to $Z_n/S_n$.

We will also need the following.
\begin{lemma}\label{lemma:stableisotypics}
    Let $V_\bullet$ and $W_\bullet$ be finitely generated $FI$-modules, and let $N$ be the sum of their stable ranges. Then for all $n\ge N$, the maps $V_n\to V_{n+1}$ and $W_n\to W_{n+1}$ (associated to $\{1,\ldots,n\}\subset\{1,\ldots,n+1\}$) induce isomorphisms
    \begin{equation*}
        V_n\otimes_{\Q[S_n]}W_n\to^\cong V_{n+1}\otimes_{\Q[S_{n+1}]}W_{n+1}
    \end{equation*}
    that are natural in both variables with respect to homomorphisms of $FI$-modules.
\end{lemma}
\begin{proof}
    By \cite[Proposition 2.3.6]{CEF1}, the tensor product $V_\bullet\otimes_{\Q} W_\bullet$ is finitely generated since $V_\bullet$ and $W_\bullet$ are. Applying the co-invariants functor, we obtain the functor
    \begin{equation*}
        n\mapsto V_n\otimes_{\Q[S_n]}W_n.
    \end{equation*}
    Because stable ranges add under tensor product (\cite[Proposition 2.3.6]{CEF1}), and because the stability degree (cf. \cite[Definition 3.1.3]{CEF1}) is less than or equal to the stable range (\cite[Proposition 3.3.3]{CEF1}), the map
    \begin{equation*}
        V_n\otimes_{\Q[S_n]}W_n\to V_{n+1}\otimes_{\Q[S_{n+1}]}W_{n+1}
    \end{equation*}
    is an isomorphism for $n\ge N$.
\end{proof}

\paragraph{Character polynomials.}  Character polynomials and their degree were defined in the introduction.  Let $\langle P,Q\rangle$ denote the inner product of $S_n$-characters.  The expectations
of character polynomials \[\Eb_{\sigma\in S_n} P_n(\sigma):=\frac{1}{n!}\sum_{\sigma\in S_n}P_n(\sigma)=\langle P_n,1\rangle\]  compute the averages of natural combinatorial statistics with respect to the uniform distribution on $S_n$.  As shown in Proposition 2.2 of \cite{CEF2}, the inner product  $\langle P_n,Q_n\rangle$ of character polynomials $P,Q\in \Q[X_1,X_2,\ldots]$ is independent of $n$ once $n\geq \deg P+\deg Q$.

One remarkable property of finitely-generated FI-modules $V$ is that the characters of the $S_n$-representations $V_n$ are, for large enough $n$, given by a single polynomial.

\begin{theorem}[{\bf Polynomiality of characters \cite{CEFN}}]
\label{theorem:FIproperties2}
Let $V$ be an FI-module over a field of characteristic $0$.   If $V$ is finitely-generated then
the characters $\chi_{V_n}$ of the $S_n$-representations $V_n$ are {\em eventually polynomial}: there exists $N\geq 0$ and a polynomial $P(X_1,\ldots ,X_r)$, for some $r>0$, so that
\begin{equation}
\label{eq:charpolthm:a}
\chi_{V_n}=P(X_1,\ldots ,X_r) \ \ \ \mbox{for all $n\geq N$.}
\end{equation}

In particular, if $Q$ is any character polynomial then $\langle \chi_{V_n},Q\rangle$ is independent of $n\geq \deg P +\deg Q$.
\end{theorem}

We note that evaluating \eqref{eq:charpolthm:a} on the identity permutation gives a polynomial $P(T)\in \Q[T]$ so that \[\dim_k V_n = P(n)\]
for all $n\geq N$.

\paragraph{\'Etale Representation Stability}
Given a co-FI-scheme $Z_\bullet$ defined over $\F_q$, its \'etale cohomology $H^i_{\et}(Z_{\bullet/\Fqbar};\Q_\ell)$  has additional structure beyond that of an $FI$-module over $\Q_\ell$.  The geometric Frobenius $\Frob_q$ gives a natural endomorphism of $Z_{\bullet/\Fqbar}$, and this gives rise to an action of $\Gal(\Fqbar/\F_q)$ on the FI-module $H^i_{\et}(Z_{\bullet/\Fqbar};\Q_\ell)$. As noted in the introduction, the  eigenvalues of $\Frob_q$ and the action of $\Gal(\Fqbar/\F_q)$ are crucial parts\footnote{As observed e.g. by Milne \cite{Mi2}, the Tate conjecture implies that the eigenvalues of $\Frob_q$ determine the $\Gal(\Fqbar/\F_q)$-action. But, this is not known at present.} of the data here.  Weaker than knowing an eigenvalue $\lambda$ of $\Frob_q$ on $H^j_{\et}(Z_{n/\Fqbar};\Q_\ell)$ is knowing its weight.  Deligne proved that $\lambda$ is an algebraic number with $|\lambda|=q^{r/2}$ for some $j\leq r\leq 2j$, with $r=j$ if $Z_n$ is smooth and proper.  The number $r$ is the {\em weight} of the eigenvalue $\lambda$.  Similarly, for $Z_\bullet$ defined over a number field $K$, the action of $\Gal(\bar{K}/K)$ on $Z_{\bullet/\bar{K}}$ induces an action on $H^i_{\et}(Z_{\bullet/\bar{K}};\Q_\ell)$, and this action is a fundamental part of the data.

In increasing order of strength, we could ask that for each $i$ there exists $D$ so that for all $n\geq D$:
\begin{enumerate}
\item The isomorphism type of $H^i_{\et}(Z_{n/\bar{K}};\Q_\ell)$ as a $\Q_\ell$-vector space does not depend on $n$;
\item in addition, the list of weights of $\Frob_q$ on $H^i_{\et}(Z_{n/\bar{K}};\Q_\ell)$ does not depend on $n$;
\item in addition, the list of eigenvalues of $\Frob_q$ on $H^i_{\et}(Z_{n/\bar{K}};\Q_\ell)$ does not depend on $n$.
\item The isomorphism type of $H^i_{\et}(Z_{n/\bar{K}};\Q_\ell)$ as a $\Gal(\bar{K}/K)$-representation does not depend on $n$.
\end{enumerate}
We have adopted the strongest of these as our definition of \'etale homological stability.

\subsection{Base change for normally compactifiable $I$-schemes}
\label{subsection:basechange}
Fix any $N$ and distinct primes $p$ and $\ell$ with $p\nmid N$.
For a scheme $X$ smooth over $\Z[1/N]$, the usual way to compare the singular cohomology $H^\ast_{sing}(X(\C);\Q)$ to the \'{e}tale cohomology
$H^\ast_{et}(X_{/\Fpbar};\Q_\ell)$ is via base change theorems.   Given a sequence of such schemes $\{X_n\}$, nne might try to prove that the vector spaces $H^i_{\et}(X_{n/\Fqbar};\Q_\ell)$ stabilize by using topology to prove homological stability of singular cohomology $H^i(X_n(\C);\Q_l)$, and then quoting the usual comparison and base change theorems that compare these two vector spaces.  However, if the $X_n$ are not projective then without more work, one has to exclude finitely many primes for each $n$ and each $i$. This list of ``bad primes'' may grow uncontrollably with $n$, and so even for a fixed $i$ it might be necessary to throw out all primes, giving no information.

Therefore we need something stronger.  As has been known for a long time, to remedy this one tries to find a smooth compactification of $X$ with normal crossings divisor.

\begin{definition}[{\bf Normally compactifiable scheme}]  A scheme $X$, smooth over $\Z[1/N]$, is \emph{normally compactifiable at $p\nmid N$} if there exists an open embedding $X\to \bar{X}$ with $\bar{X}$ smooth and projective over $\Z[1/N]$, and $\bar{X}-X$ a normal crossings divisor with good reduction at $p$, i.e. $(\bar{X}-X)|_{\F_p}$ is a union of smooth, projective varieties over $\F_p$, with all intersections transverse.
\end{definition}

We will need the notion of a sequence of schemes that can be compactified in a uniform way.

\begin{definition}[{\bf Normally compactifiable $I$-scheme}]
    Let $I$ be a category. A \emph{smooth (projective) $I$-scheme} is a functor $X\colon I\to \Schemes$ consisting of schemes, smooth (and projective) over $\Z[1/N]$.  A smooth $I$-scheme is \emph{normally compactifiable at $p\nmid N$} if there exists a smooth projective $I$-scheme $\bar{X}\colon I\to \Schemes$, and a natural transformation $X\to\bar{X}$ such that, for all $i\in I$, $X_i\to\bar{X}_i$ is an open embedding and $\bar{X}_i-X_i$ is a normal crossings divisor with good reduction at $p$.
\end{definition}

With this definition in hand we can state the appropriate base change theorem for $I$-schemes. We will apply this later in the paper to the category $FI^{\op}$.

\begin{theorem}[{\bf Base change for $I$-schemes}]
\label{thm:basechange}
    Let $X\colon I\to\Schemes$ be a smooth $I$-scheme which is normally compactifiable at $p\nmid N$. Then, for all morphisms $i\to j$ in $I$, there exists a commuting square of ring homomorphisms
	\begin{equation*}
        \xymatrix{
            H^\ast_{et}(X_{j/\Fpbar};\Z_\ell) \ar[rr]^\cong \ar[d] && H^\ast_{sing}(X_j(\C);\Z_\ell) \ar[d]\\
            H^\ast_{et}(X_{i/\Fpbar};\Z_\ell) \ar[rr]^\cong && H^\ast_{sing}(X_i(\C);\Z_\ell)}
    \end{equation*}
    with horizontal maps isomorphisms.
\end{theorem}
\begin{proof}
    This theorem is essentially an $I$-scheme version of \cite[Proposition 7.7]{EVW}. Suppose that we are given $X\colon I\to \Schemes$ as above.  We can restrict the diagram $X$ to an \'etale neighborhood of $p$ in $\Spec(\Z[1/N])$, or even to the strict henselization of the local ring at $p$. The horizontal homomorphisms in the statement of the theorem are now the specialization morphisms from the cohomology of the geometric special fibers of $X$ at $p$ to the geometric generic fiber of $X$ at $\Q$. As explained in the proof of \cite[Proposition 7.7]{EVW}, for each $i\in I$, the Poincar\`e dual of the specialization map
    \begin{equation*}
        H^\ast_{et}(X_{i/\Fpbar};\Z_\ell)\to H^\ast_{et}(X_{i/\C};\Z_\ell)
    \end{equation*}
    is an isomorphism because $X_i$ admits a smooth compactification $X_i\into \bar{X_i}$ with $\bar{X_i}-X_i$ a normal crossings divisor with good reduction at $p$. By construction, this map is natural (cf. e.g. \cite[Ch. III, \S 3]{FK}), so for all $i\to j$ in $I$, there is a commuting square
    \begin{equation*}
        \xymatrix{
            H^\ast_{et}(X_{j/\Fpbar};\Z_\ell) \ar[rr]^\cong \ar[d] && H^\ast_{et}(X_{j/\C};\Z_\ell) \ar[d] \\
            H^\ast_{et}(X_{i/\Fpbar};\Z_\ell) \ar[rr]^\cong && H^\ast_{et}(X_{i/\C};\Z_\ell)
        }
    \end{equation*}
    To obtain the desired square, we compose this with the commuting square
    \begin{equation*}
        \xymatrix{
            H^\ast_{et}(X_{j/\C};\Z_\ell) \ar[rr]^\cong \ar[d] && H^\ast_{sing}(X_j(\C);\Z_\ell) \ar[d] \\
            H^\ast_{et}(X_{i/\C};\Z_\ell) \ar[rr]^\cong && H^\ast_{sing}(X_i(\C);\Z_\ell)
        }
    \end{equation*}
    given by Artin's Comparison theorem \cite[Theorem 21.1]{Mi}.
\end{proof}

\subsection{\'{E}tale representation stability}
\label{subsection:repstab}
Let $Z$ be a co-FI scheme smooth over $\Z[1/N]$ for some fixed $N$, with geometrically connected fibers. Let  $p\nmid N$ be prime, and let $\ell\neq p$ be a prime. For each $i\geq 0$, the \'{e}tale cohomology $H^i_{et}(Z_{/\Fpbar};\Zb_\ell)$ is an FI-module.  In addition, for each $q=p^d$, the Frobenius $\Frob_q$ acts on each
$H^i_{et}({Z_n}_{/\Fpbar};\Zb_\ell)$, endowing it with the structure of a $\Gal(\Fpbar/\F_q)$-module.  The $S_n$-action on  $H^i_{et}({Z_n}_{/\Fpbar};\Zb_\ell)$ coming from its structure as an FI-module commutes with the action of $\Gal(\Fpbar/\F_q)$, as do all automorphisms of $Z_n$. Similarly, for any number field $K$, the action of $\Gal(\Kbar/K)$ on $H^i_{\et}(Z_{/\Kbar};\Q_\ell)$ commutes with the $FI$-structure.

This discussion shows that, for $K$ a number field or a finite field of characteristic prime to $N$, $H^i_{et}(Z_{/\Kbar};\Q_\ell)$ is an {\em $\Gal(\Kbar/K)$-FI-module}; that is, an FI-module equipped with an action of $\Gal(\Kbar/K)$ by FI-automorphisms.  We have the corresponding notions of finitely generated $\Gal(\Kbar/K)$-FI-module: there is a finite set $S\subset \coprod_n H^i_{et}(Z_n\ _{/\Kbar};\Q_\ell)$ so that no proper sub-FI-module of $H^i_{et}(Z_{/\overline{K}};\Q_\ell)$ contains $S$.

\begin{definition}[{\bf \'{E}tale representation stability}]
    We say that a sequence $Z_n$ of $\Gal(\Kbar/K)$-modules satisfies {\em \'{e}tale representation stability} if $\{1,\ldots ,n\}\mapsto Z_n$ is a finitely generated $\Gal(\Kbar/K)$-FI-module.
\end{definition}

Theorem C in \cite{CEFN} gives an inductive description of finitely-generated FI-modules $V$ over any Noetherian ring $R$.  Namely, there exists $D\geq 0$ such that for all $n\in \Nb$, there is a natural isomorphism
\begin{equation}
\label{eq:colim2}
V_n\cong\ \colim_S V_S
\end{equation}
where the colimit is taken over the poset of all subsets $S\subset \{1,\ldots ,n\}$ such that
$|S|\leq D$.  If $V$ is a finitely-generated $\Gal(\Kbar/K)$-FI-module then \eqref{eq:colim2} gives is an isomorphism of $\Gal(\Kbar/K)$-modules. Thus \eqref{eq:colim2} gives, for each $n\geq D$, a recipe for building the $\Gal(\Kbar/K)$-representation $V_n$ from a fixed finite collection of $\Gal(\Kbar/K)$-representations.

\paragraph{\'{E}tale representation stability for products and configuration spaces. }
For any scheme $X$ and any finite set $S$, let $X^S$ denote the scheme of functions $S\to X$, i.e. the scheme whose $R$ points are maps of maps of sets $S\to \Hom(\Spec(R),X)$.  Any injection $f:T\to S$ induces a projection $f^*:X^S\to X^T$ by restriction. The functor $S\mapsto X^S$ is thus a co-FI scheme, which we denote by $X^\bullet$. Attached to $X$ we also have the associated {\em configuration space} $\PConf_n(X)$ of {ordered $n$-tuples} in $X$, defined by:

\[\PConf_n(X):=\{(x_1,\ldots ,x_n)\in X^n: x_i\neq x_j \ \forall i\neq j\}=X^n\setminus\Delta \]
where $\Delta$ is the fat diagonal and where we write $x\in X$ to denote an arbitrary $R$-point of $X$.   The group $S_n$ acts freely on $\PConf_n$ by permuting the coordinates. The quotient $\UConf_n(X):=\PConf_n(X)/S_n$ is the {\em configuration space of unordered $n$-tuples} of points in $X$.

We can identify $\PConf_n(X)$ with the subspace $\Emb(S,X)$ of $X^S$ consisting of embeddings.  If $f:S\to T$ is an injection of finite sets and $\phi:T\to X$ is an embedding, then $\phi\circ f:S\to X$ is an embedding.  Thus $S\mapsto \PConf_{|S|}(X)$ is a co-FI space, denoted $\PConf_\bullet(X)$.  It follows that $H^i(\PConf_\bullet(X);\Q_\ell)$ is an FI-module for any $i\geq 0$.

With this setup, we can now prove Theorem A from the introduction.  The proof also gives the following result.

\begin{theorem}
\label{theorem:extra}
Theorem A with $\PConf_n(Y)$ (resp.\ $\UConf_n(Y)$) replaced by $Y^n$ (resp.\ $\Sym^n(Y)$) holds.
\end{theorem}

\begin{proof}[Proof of Theorem A and Theorem~\ref{theorem:extra}]
Since $X$ is normally compactifiable, $H^*(X(\C);\C)$ is finite-dimensional. Since $X$ has geometrically connected fibers, $X(\C)$ is connected. Since $X(\C)$ is a complex manifold, it is also orientable and, assuming that $X$ is not a point (for otherwise the theorem is trivial), $X(\C)$ has real dimension at least $2$.  Thus $X$ satisfies the hypotheses of Theorem 6.2.1 of \cite{CEF1}, which gives that the FI-module $H^i(\PConf_\bullet(X(\C));\C)$ is finitely generated for each $i\geq 0$.  For the claimed
stable ranges and degree of character polynomial, see Theorem 6.3.1 of \cite{CEF1}.  Proposition 6.1.2 of \cite{CEF1} gives the same conclusion for the FI-module $H^i(X(\C)^\bullet;\Q)$; see Theorem 4.1.7 and Remark 6.1.3 of \cite{CEF1} for the stable range and degree of the character polynomial.  Note here that this improved stable range comes from the fact that $X^\bullet$ is a so-called FI\#-module.

By Artin's Comparison theorem, for $Z_\bullet$ equal to $X^\bullet$ or $\PConf_\bullet(X)$ :
\begin{equation*}
    H^i_{\et}(Z_{\bullet/\bar{\Q}};\Q_\ell)\otimes_{\Q_\ell}\C \cong H^i(Z_{\bullet}(\C);\C)
\end{equation*}
as FI-modules, and thus $H^i_{\et}(Z_{\bullet/\bar{\Q}};\Q_\ell)$ is finitely generated.  Moreover, because, $Z_\bullet$ is a co-FI-scheme over $\Zb[1/N]$, for each number field $K$  the group $\Gal(\Kbar/K)$ acts on the base-change $Z_{\bullet/\bar{\Q}}$ and this action commutes with the co-FI-structure.  We conclude, for each $i$, that $H^i_{\et}(Z_{\bullet/\bar{\Q}};\Q_\ell)$ is a finitely generated co-$\Gal(\Kbar/K)$-FI-module, as claimed.

Now let $K=\F_q$ for $q=p^d$ with $p\nmid N$. The scheme $X^n$ is normally compactifiable since $X$ is, and thus $X^\bullet$ is a normally compactifiable, smooth co-FI scheme.  Results of Fulton and MacPherson \cite{FM} imply that $\PConf_\bullet(X)$ is also a normally compactifiable, smooth co-FI scheme.  Indeed, in \emph{loc. cit.} they construct a co-$FI$-scheme $J\mapsto X[J]$ (cf. \cite[p. 184 (4)]{FM}), equipped with a natural open embedding $\PConf_J(X)\to X[J]$, and such that the complement $X[J]-\PConf_J(X)$ is a normal crossings divisor \cite[Theorem 3]{FM}, with good reduction at $p$ for any prime $p\nmid N$ (cf. \cite[bottom of p. 188]{FM}). Base change for FI-schemes (Theorem~\ref{thm:basechange}) implies that $H_{\et}^i(X^\bullet _{/\Fqbar};\Q_\ell)$ and $H_{\et}^i({\PConf_\bullet (X)}_{/\Fqbar};\Q_\ell)$ are finitely-generated
$\Gal(\Fqbar/\F_q)$-FI-modules.  Now apply
Theorem~\ref{theorem:FIproperties}.
\end{proof}

\section{Convergent Cohomology}
\label{section:convergent:cohomology}
In this section we provide the necessary bounds for the ``representation unstable cohomology'' of $X^n$ and of $\PConf_n(X)$ that will be necessary for the arithmetic applications in Section~\ref{section:arithmetic:stats}.  This comprises one of the main technical contributions of this paper.

\subsection{Definition of convergent cohomology}

A function $F:\Nb\to\Nb$ {\em has exponential growth rate $\lambda$} if
\begin{equation}
\label{eq:growth}
\lim_{n\to\infty}\frac{\displaystyle \log f(n)}{\displaystyle n}=\lambda.
\end{equation}
If \eqref{eq:growth} holds with $\lambda=0$, we say that $F$ has {\em sub-exponential growth}.

Let $Z$ be a co-FI-scheme over $\Z[1/N]$.  For each $i\geq 0$, let $H^i(Z_n)$ denote
either the singular cohomology $H^i(Z_n(\C);\Q)$ or the \'{e}tale cohomology $H^i_{\et}(Z_{n/\Kbar};\Q_\ell)$ for $K$ a number field or finite field of characteristic prime to $N$.    In each case $H^i(Z)$ is an FI-module (over $\Q$ and $\Q_\ell$, respectively).  For any class function $P$ on $S_n$, denote by $\langle P,H^i(Z_n)\rangle$ the inner product of (the character of) $H^i(Z_n)$ with $P$.

In order to compute arithmetic statistics for a co-FI scheme $Z$, one needs to control the ``representation unstable'' cohomology of $Z$; see \S\ref{section:arithmetic:stats}.
More precisely, one needs to prove one of the following two properties, which were shown to be equivalent in \cite[\S 3]{CEF2}:

\begin{enumerate}
\item For each $0\leq a\leq n$ there is a function $F_a(i)$, subexponential in $i$ and not depending on $n$, so that:
\begin{equation}
\label{eq:conv5}
\dim H^i(Z_n)^{S_{n-a}}\leq F_a(i) \text{ for all }n\text{ and }i.
\end{equation}

\item  For each character polynomial $P\in \Q[X_1,X_2,\ldots]$ there exists a function $F_P(i)$, subexponential in $i$ and not depending on $n$, such that:
\begin{equation}
\label{eq:conv6}
\big|\langle P,H^i(Z_n)\rangle\big|\leq F_P(i) \text{ for all }n\text{ and }i.
\end{equation}
\end{enumerate}

It is crucial that these bounds hold independently of $n$.  While the second condition is the one that applies to arithmetic statistics (see \S\ref{section:arithmetic:stats} below), it is quite difficult to check. Thus the equivalence with the first condition is quite useful.

\begin{definition}[{\bf Convergent cohomology}]
\label{def:convergent}
We say that the co-FI scheme $Z$ has {\em convergent (singular or \'{e}tale) cohomology} if either of the two equivalent properties 1 or 2 in equations \eqref{eq:conv5} or \eqref{eq:conv6}
holds.  If these properties hold with $F_P(i)$ having exponential growth $0<\lambda<\infty$, we say that $Z$ has {\em weakly convergent cohomology with convergence rate $\lambda$}.
\end{definition}

These kinds of bounds are typically not easy to prove.   In \cite{EVW} this is accomplished \footnote{\cite{EVW} only needs to deal with the classical, not representation stable, case; that is, the $a=0$ case.} for the cohomology of certain Hurwitz spaces by obtaining an exponential upper bound for the number of $i$-cells, via an explicit cell decomposition. In \cite{CEF2} such bounds for the example $H^i(\PConf_n(\C);\Q)$ are obtained by a detailed knowledge of these $S_n$-representations. The rest of this section is devoted to giving such bounds for two natural classes of co-FI schemes.  We then apply this in \S\ref{section:arithmetic:stats} to arithmetic statistics for $\F_q$-points on these schemes.

\subsection{Subexponential bounds on Betti numbers of symmetric products}
Let $X$ be a topological space.  The $n$-fold cartesian product $X^n$ is endowed with a natural action of the symmetric group $S_n$, given by permuting the factors.  The quotient $\Sym^nX:=X^n/S_n$ is called the the {\em $n^{\rm th}$ symmetric product} of $X$.   Set $b_i(X)=\dim H^i(X;\Q)=\dim H_i(X;\Q)$, the equality following from the universal coefficient theorem.
\begin{theorem}[{\bf Growth of Betti numbers of symmetric products}]
\label{thm:symbound}
    Let $X$ be a connected space such that $\dim H^\ast(X;\Qb)<\infty$.   Then there exist constants $K,L>0$ so that for each $i\geq 0$, and for all $n\geq 1$:
    \begin{equation*}
      b_i(\Sym^n(X);\Q)\leq Ke^{L\sqrt{i}}.
    \end{equation*}
\end{theorem}

\begin{proof} We prove the theorem in a series of steps.

\bigskip
\noindent
{\bf Step 1 (Bounding the unstable range): } If $m<n$ then $b_i(\Sym^m(X))\le b_i(\Sym^n(X))$; further, $b_i(\Sym^n(X))=b_n(\Sym^n(X))$ for all $i\geq n$.

    To see the statement for $m<n$, observe that for any graded vector space $V$ (over a field of characteristic 0), a choice of ``unit'' $1\in V_0$, determines an injection
    \begin{align*}
        \Sym^m(V)_i&\to\Sym^n(V)_i\\
        \vec{v}&\mapsto\vec{v}\otimes 1\otimes\cdots\otimes 1
    \end{align*}
    In particular, $\dim(\Sym^m(V)_i)\le \dim(\Sym^n(V)_i)$ for all $m<n$ and $i$. Considering $V=H^\ast(X;\Qb)$, K\"unneth and transfer imply that
    \begin{equation*}
        \Sym^n(H^\ast(X;\Qb))_i\cong H^i(\Sym^n(X);\Qb)
    \end{equation*}
    and the first part of the claim follows.  For the second, we note that for a graded vector space $V=V_0\oplus\cdots\oplus V_m$, with $V_0=\Qb$, we have
    \begin{equation*}
        \Sym^n(V)\cong\bigoplus_{a_0+\cdots a_m=n}\bigotimes_{j=1}^m \Sym^{a_j}(V_j)
    \end{equation*}
    and thus
    \begin{equation*}
        \Sym^n(V)_i\cong \bigoplus_{(a_1,\ldots,a_m)}\bigotimes_{j=1}^m\Sym^{a_j}(V_j)
    \end{equation*}
    where the direct sum is over partitions $a_1+2a_2+\cdots ma_m=i$ such that $a_1+\cdots+a_m = n$. In particular, the number of pieces in the partition is at most $n$, and since for any $i$, the largest number of pieces in any partition is $i$, we see that for $n\ge i$, the direct sum is independent of $n$. We conclude the claim by taking $V=H^\ast(X;\Qb)$ and invoking K\"unneth and transfer as above.

\medskip
\noindent
{\bf Step 2 (Generating functions and the residue formula): }
Let $m$ be the highest degree for which $H^m(X;\Qb)\neq 0$, let $C=\max_i b_i(X)$, and let $D=C\cdot m$.
    Write $B_i:=b_i(X)$. Macdonald \cite{Mac} showed that $b_i(\Sym^n(X))$ is the coefficient of $x^it^n$ in the power series expansion of the rational function
    \begin{equation*}
        F_X(x,t)=\frac{(1+xt)^{B_1}(1+x^3t)^{B_3}(1+x^5t)^{B_5}\cdots}{(1-t)(1-x^2t)^{B_2}(1-x^4t)^{B_4}\cdots}
    \end{equation*}
    Step 1 implies that $b_i(\Sym^n(X))\le b_i(\Sym^i(X))$ for all $n$, and thus a bound is given by the coefficient of $x^it^i$ in the power series above, i.e. it suffices to consider the generating function
    \begin{equation*}
        \diag(F_X)(z)=\sum_{i=0}^\infty b_i(\Sym^i(X))z^i.
    \end{equation*}
    Because this power series has no poles for $x,t<1$, \cite[Theorem 1]{HK} implies that for all $z$ with $|z|<1$, we have
    \begin{equation*}
        \diag(F_X)(z)=\frac{1}{2\pi i}\int_{|s|=1} \frac{(1+z)^{B_1}(1+s^2z)^{B_3}(1+s^4z)^{B_5}\cdots}{(s-z)(1-sz)^{B_2}(1-s^3z)^{B_4}\cdots}ds
    \end{equation*}
    Note that, because $|z|<1$, the integrand has a single pole inside the unit circle $|s|=1$ at $s=z$. Therefore, by the Residue Formula:
    \begin{align*}
        \diag(F_X)(z)&=\res_{s=z} \frac{1}{s-z}\left(\frac{(1+z)^{B_1}(1+s^2z)^{B_3}\cdots}{(1-sz)^{B_2}(1-s^3z)^{B_4}\cdots}\right)\\
        &=\frac{(1+z)^{B_1}(1+z^3)^{B_3}\cdots}{(1-z^2)^{B_2}(1-z^4)^{B_4}\cdots}
    \end{align*}
    Therefore, our generating function is given by
    \begin{equation}\label{gen}
    \diag(F_X)(z)=    \prod_{\ell=1}^{\lceil\frac{m}{2}\rceil}(1+z^{2\ell-1})^{B_{2\ell-1}}\left(\sum_{k=0}^\infty z^{2\ell k}\right)^{B_{2\ell}}.
    \end{equation}

    \medskip
    \noindent
    {\bf Step 3 (Asymptotics of the coefficients): }
    We claim that the coefficient of $z^i$ in the expansion of \eqref{gen} is less than or equal to the coefficient of $z^i$ in
    \begin{equation}\label{genbound}
        \left(\sum_{k=0}^\infty z^k\right)^D.
    \end{equation}
    To see this, let $f(z)=\sum_{i=0}^\infty f_i z^i$ and $g(z)=\sum_{i=0}^\infty g_iz^i$ be power series with non-negative coefficients.  Let us say that $f(z)\le g(z)$ if, for all $i\ge 0$, $f_i\le g_i$. Then for any other power series $h(z)$ with non-negative integer coefficients, it is easy to see that
    \begin{equation*}
        f(z)h(z)\le g(z)h(z).
    \end{equation*}
    Taking $f(z)=1$ and $h(z)=g(z)$, this implies that for all $A<B$,
    \begin{equation*}
        g(z)^A\le g(z)^{B}.
    \end{equation*}
    Applying this to \eqref{gen}, we have
    \begin{align*}
         \prod_{\ell=1}^{\lceil\frac{m}{2}\rceil}(1+z^{2\ell-1})^{B_{2\ell-1}}\left(\sum_{k=0}^\infty z^{2\ell k}\right)^{B_{2\ell}}&\le \prod_{\ell=1}^{\lceil\frac{m}{2}\rceil}(1+z^{2\ell-1})^C\left(\sum_{k=0}^\infty z^{2\ell k}\right)^{C}\\
         &\le \prod_{\ell=1}^{\lceil\frac{m}{2}\rceil}\left(\sum_{k=0}^\infty z^k\right)^{2C}\\
         &=\eqref{genbound}.
    \end{align*}
    We now conclude the proof of the theorem as follows. Given a partition $J\vdash i$, denote by $|J|$ the number of blocks in the partition $J$. With this notation, for $i>>D$, the coefficient of $z^i$ in the expansion of \eqref{genbound} is at most
    \begin{equation}
    \label{eq:upper1}
        \sum_{j=1}^D\binom{D}{j}j!|\{J\vdash i~|~|J|=j\}|\le \sum_{j=1}^D \binom{D}{j}j!|\{J\vdash i\}|.
    \end{equation}
    The bound on the left-hand side comes from observing that the coefficient of $z^i$ is the number of all possible ways of choosing $1\le j\le D$ factors $z^{i_1},\ldots,z^{i_j}$ with $i_a>0$ from the $D$ factors in \eqref{genbound} such that $i_1+\ldots+i_j=i$.  In other words, this amounts to choosing $j$ factors which have positive powers of $z$ (from the $D$ total factors) along with an ordered partition of $i$ into $j$ pieces.  Since partitions are naturally unordered, the set of ordered partitions with $j$ pieces has at most $j!$ elements for each unordered partition with $j$ pieces.

The Hardy-Ramanujan asymptotic for the number $|\{J\vdash i\}|$ of partitions of $i$ gives
$C_1,C_2>0$ so that
\begin{equation}
\label{eq:HR}
|\{J\vdash i\}| \leq C_1e^{C_2\sqrt{i}}.
\end{equation}

Applying this to the upper bound \eqref{eq:upper1} gives

    \begin{align*}
        \eqref{eq:upper1}&\le C_1e^{C_2\sqrt{i}}  \left(\sum_{j=1}^D\binom{D}{j}j!\right)\\
        &\le D!  C_1e^{C_2\sqrt{i}}   \left(\sum_{j=1}^D \binom{D}{j}\right)\\
        &\le 2^D D! C_1e^{C_2\sqrt{i}}.
    \end{align*}

\end{proof}

\paragraph{Consequence: bounding the representation unstable cohomology of products.}
The following corollary is also a key ingredient in bounding the representation unstable cohomology of configuration spaces.

\begin{corollary}
\label{corollary:bounding:products}
Let $X$ be a connected space such that $\dim H^\ast(X;\Qb)<\infty$.   For each $0\leq a
\leq n$ there are constants $K,L>0$ so that for each $i\geq 0$, and for all $n\geq 1$:
\[\dim (H^i(X^n;\Q)^{S_{n-a}})\leq Ke^{L\sqrt{i}}.\]
\end{corollary}

\begin{proof}
Since the action $S_{n-a}$ leaves invariant the first $n-a$ factors of $X^n$ and acts as the identity on the last $a$ factors, there is, for each $i\geq 0$, an isomorphism:
\begin{equation}
\label{eq:prods1}
H^i(X^n)^{S_{n-a}}=\bigoplus_{p+q=i}H^p(X^{n-a})^{S_{n-a}}\otimes H^q(X^a).
\end{equation}
Since this sum has $i+1$ terms, it suffices to bound the dimension of each summand by $Ke^{L\sqrt{i}}$.  Since $\dim H^\ast(X;\Qb)<\infty$ and since $a$ is fixed,  there is a constant $C$, not depending on $q$, so that $\dim(H^q(X^a))\leq C$.  It thus suffices to bound each $H^p(X^{n-a})^{S_{n-a}}$ by $Ke^{L\sqrt{i}}$.  Now transfer together with Theorem~\ref{thm:symbound} gives
$\dim(H^p(X^{n-a})^{S_{n-a}})\leq Ke^{L\sqrt{p}}\leq Ke^{L\sqrt{i}}$
for some constants $K,L$, since $p\leq i$.
\end{proof}

\subsection{Bounding the representation unstable cohomology of configuration spaces}

We build on the subexponential upper bounds for products in the last section to prove the corresponding result for configuration spaces.

\begin{theorem}
\label{theorem:subexp2}
Let $X$ be a smooth, orientable manifold with $\dim(H^*(X;\Q))<\infty$ (e.g. $X$ compact).
Then the co-FI manifold $\PConf_\bullet(X)$ has convergent singular cohomology.
\end{theorem}

\begin{proof}
   Fix $a\geq 0$.  Denote by $S_{n-a}$ the subgroup $S_{n-a}\times 1\subset S_n$.  We will prove that there is a function $F_a(i)$, subexponential in $i$, so that:
    \begin{equation*}
        \dim \left(H^i(\PConf_n(X);\Qb)\right)^{S_{n-a}}\le F_a(i)
    \end{equation*}
    for all $i\geq 0$.  Let $m$ be the real dimension of $X$, and denote by $A(n,m)$ the graded commutative algebra
    \begin{equation*}
        A(n,m):=\Qb[\{G_{ab}\}_{1\le a\neq b\le n}]/I
    \end{equation*}
    where $|G_{ab}|=2m-1$ and $I$ is the ideal generated by the elements
    \begin{align*}
        &G_{ab}-G_{ba}\\
        &G_{ab}G_{ac}+G_{bc}G_{ba}+G_{ca}G_{cb}
    \end{align*}
    for $a<b<c$ distinct. The group $S_n$ acts on $A(n,m)$ via $\sigma\cdot G_{ab}:=G_{\sigma(a)\sigma(b)}$. Totaro \cite[Theorem 4]{To} has shown that $H^\ast(\PConf_n(X);\Qb)$ is isomorphic, as a graded $S_n$-representation, to a sub-quotient of
    \begin{equation*}
        H^\ast(X^n;\Qb)\otimes A(n,m),
    \end{equation*}
    with the natural action on each factor.

    Note that, for any short exact sequence of $S_n$-representations
    \begin{equation*}
        0\to V_0\to V_1\to V_2\to 0
    \end{equation*}
    over a field of characteristic 0, there exists an $S_n$-equivariant splitting
    \begin{equation*}
        V_1\cong V_0\oplus V_2.
    \end{equation*}
    In particular,
    \begin{equation*}
        \dim V_1^{S_n}= \dim V_0^{S_n}+\dim V_2^{S_n},
    \end{equation*}
    and, more generally, if $V$ is any sub-quotient of an $S_n$-representation $W$, we have
    \begin{equation*}
        \dim V^{S_n}\le \dim W^{S_n}.
    \end{equation*}

 Let $V$ and $W$ be any two $S_n$-representations.    The identity $\dim(V^{S_n})=\langle \chi_V,1\rangle$ and the Cauchy-Schwarz inequality give:
    \begin{align*}
        \dim (V\otimes W)^{S_n}&=\langle \chi_{V\otimes W},1\rangle\\
        &=\frac{1}{n!}\sum_{\sigma\in S_n} \chi_V(\sigma)\chi_W(\sigma)\\
        &\le \frac{1}{n!}\sqrt{\left(\sum_{\sigma\in S_n} \chi_V(\sigma)^2\right)\left(\sum_{\sigma\in S_n}\chi_W(\sigma)^2\right)}\\
        &=\sqrt{\langle \chi_{V^{\otimes 2}},1\rangle\langle\chi_{W^{\otimes 2}},1\rangle}\\
        &=\sqrt{\dim((V^{\otimes 2})^{S_n})\dim((W^{\otimes 2})^{S_n})}.
    \end{align*}
    Specializing to our setting, we conclude that it suffices to show that
    \begin{equation*}
        (\dim \left(H^i(X(\Cb)^n;\Qb)^{\otimes 2}\right)^{S_{n-a}})\cdot (\dim (A(n,m)^{\otimes 2})^{S_{n-a}})\le F_a(i)
    \end{equation*}
for some $F_a(i)$ subexponential in $i$.  For the first factor, by K\"unneth and the definition of the action, we have
    \begin{align*}
        (H^i(X^n;\Qb)^{\otimes 2})^{S_{n-a}}&\subset H^{2i}(X^n\times X^n;\Qb)^{S_{n-a}}\\
        &\cong H^{2i}((X\times X)^n;\Qb)^{S_{n-a}}\\
        &\cong \bigoplus_{j=0}^{2i} H^j((X\times X)^{n-a};\Qb)^{S_{n-a}}\otimes H^{2i-j}((X\times X)^a;\Qb).
    \end{align*}
    By transfer, this is isomorphic to
    \begin{align*}
        \bigoplus_{j=0}^{2i} H^j(\Sym^{n-a}(X\times X);\Qb)\otimes H^{2i-j}((X\times X)^a;\Qb)
    \end{align*}
    Let $C=\max_i b_i(X\times X)$, and let $D=C\cdot 2m$. By K\"unneth, for all $j<2i$,
    \begin{equation*}
       \dim H^{2i-j}(X^{2a};\Qb)\le (2i-j)C^a.
    \end{equation*}
    Combining this with Theorem \ref{thm:symbound}, we see that
    \begin{align*}
        \dim (H^i(X^n;\Qb)^{\otimes 2})^{S_{n-a}}&\le \sum_{j=0}^{2i} \dim \left(H^j(\Sym^{n-a}(X\times X);\Qb)\right)(2i-j)C^a\\
        &\le \sum_{j=0}^{2i} 2^D D! C_1 e^{C_2\sqrt{j}}(2i-j)C^a\\
        &\le 2^D D!C^a(2i)^2C_1e^{C_2\sqrt{2i}}.
    \end{align*}

It remains to bound $\dim (A(n,m)\otimes A(n,m))_i^{S_{n-a}}$. Well,

\begin{equation}
\label{eq:anm}
(A(n,m)\otimes A(n,m))_i=\bigoplus_{p+q=i}(A(n,m)_p\otimes  A(n,m)_q)
\end{equation}

Since the right-hand side of \eqref{eq:anm} has at most $2i$ terms, it suffices to bound each $[A(n,m)_p\otimes  A(n,m)_q]^{S_{n-a}}$.  By the Cauchy-Schwartz inequality, as above, it suffices to bound $[A(n,m)_p\otimes A(n,m)_p]^{S_{n-a}}$ for each $1\leq p\leq i$.  To obtain this bound, first note that the algebra $A(n,m)$ is isomorphic to $A(n,2)$ via an isomorphism that takes the $p^{\rm th}$ graded piece of $A(n,2)$ to the
$(2m-1)p^{\rm th}$ graded piece of $A(n,m)$. Since $m$ is fixed and so $2m-1$ is fixed,
it suffices to bound $[A(n,2)_p\otimes A(n,2)_p]^{S_{n-a}}$ in terms of $i$, for each $1\leq p\leq i$.

Lehrer-Solomon \cite{LS} give an explicit description of $A(n,2)$ as a sum
of induced representations
\[A(n,2)_p=\bigoplus_\mu \Ind_{Z(c_\mu)}^{S_n}(\xi_\mu)\]
where $\mu$ runs over the set of conjugacy classes in $S_n$ of permutations having $n-p$ cycles, $c_\mu$ is any element of the conjugacy class $\mu$, and $\xi_\mu$ is a one-dimensional character of the centralizer  $Z(c_\mu)$ of $c_\mu$ in $S_n$ (we will not
need an explicit description of $\xi_\mu$).   It follows that

\begin{equation}
\label{eq:tensind1}
(A(n,2)_p\otimes A(n,2)_p)^{S_{n-a}}=\bigoplus_{\mu,\nu}[ \Ind_{Z(c_\mu)}^{S_n}(\xi_\mu)\otimes \Ind_{Z(c_\nu)}^{S_n}(\xi_\nu)]^{S_{n-a}}
\end{equation}
where $\nu$ is defined similarly to $\mu$.  The summands contributing to the first (resp. second) $A(n,2)_p$ factor in \eqref{eq:tensind1} correspond to conjugacy classes $c_\mu$ (resp. $c_\nu$) in $S_n$ decomposing into $n-p$ cycles.  The number of such conjugacy classes is in bijection with the set of partitions of $p$, which is less than the number of partitions of $i$ since $p\leq i$.  Thus the number of terms in the sum on the right-hand side of \eqref{eq:tensind1} is, by \eqref{eq:HR},  at most $ [C_1e^{C_2\sqrt{i}}]^2=C_1^2e^{2C_2\sqrt{i}}$.  As this is subexponential in $i$, it suffices to bound the dimension of $[ \Ind_{Z(c_\mu)}^{S_n}(\xi_\mu)\otimes \Ind_{Z(c_\nu)}^{S_n}(\xi_\nu)]^{S_{n-a}}$.

Now, a permutation $c_\mu$ decomposing into $n-p$ cycles must have at least $n-2p$ fixed points. This implies that the centralizer $Z(c_\mu)$ contains the subgroup $S_{n-2p}$, and thus $S_{n-2i}$ since $p\leq i$.   It follows that $\Ind_{Z(c_\mu)}^{S_n}(\xi_\mu)$ is a subrepresentation of $\Ind_{S_{n-2i}}^{S_n}(\xi_\mu)$.  Thus
\begin{equation}
\label{eq:tensind2}
[ \Ind_{Z(c_\mu)}^{S_n}(\xi_\mu)\otimes \Ind_{Z(c_\nu)}^{S_n}(\xi_\nu)]^{S_{n-a}} \subset
[\Ind_{S_{n-2i}}^{S_n}(\xi_\mu)\otimes \Ind_{S_{n-2i}}^{S_n}(\xi_\nu)]^{S_{n-a}}
\end{equation}

Let $\chi_\mu$ and $\chi_\nu$ denote the characters of $\xi_\mu$ and $\xi_\nu$, respectively.
The right-hand side of \eqref{eq:tensind2} consists of the set of bilinear functions
$f:S_n\times S_n\to \C$ satisfying
\[f(\sigma\cdot g,\tau\cdot h)=
\chi_\mu(\sigma)\chi_\nu(\tau)f(g,h)
 \ \ \ \forall \sigma,\tau\in S_{n-2i}\ \ \text{and} \ \ \forall g,h\in S_n\]
 and
 \[f(g\cdot\beta,h\cdot\beta)=f(g,h) \ \ \ \forall\beta\in S_{n-a}\ \ \text{and} \ \ \forall g,h\in S_n.\]
 It follows that the dimension of this vector space is at most the number of double cosets
 \[S_{n-a}\backslash [S_n/S_{n-2i} \times S_n/S_{n-2i}].\]

We claim that this number is polynomial in $i$. Indeed, it is equal to the number of maps $f\colon \{1,\ldots,a\}\to \{1,\ldots,2i,\star\}\times \{1,\ldots,2i,\star\}$ such that $|f^{-1}(j,k)|\leq 1$ and $|f^{-1}(j,\star)|,|f^{-1}(\star,k)|\leq (n-2i)^2$.  Since $a$ is fixed, this number is bounded by a constant times the number of subsets of $\{1,\ldots,2i,\star\}\times \{1,\ldots,2i,\star\}$ of size $\leq a$, which is $O(i^{2a})$.  This completes the proof of Theorem~\ref{theorem:subexp2}.
\end{proof}

\section{Stability of arithmetic statistics}
\label{section:arithmetic:stats}

Throughout this section we will fix a prime power $q=p^d$ and a prime $\ell$ not divisible by $p$.

\subsection{Point counting and \'{e}tale cohomology}
\label{section:background}

Let $Y$ be a scheme of finite type (not necessarily smooth) over $\Z[1/N]$.  We can base change to $\Spec(\F_p)$ for any prime $p\nmid N$, and for any positive power $q=p^d$ we can consider both the $\F_q$-points as well as the $\Fqbar$-points of $Y$, where $\Fqbar$ is the algebraic closure of $\F_q$.  The {\em geometric Frobenius} morphism $\Frob_q\colon Y\to Y$ acts on $Y(\Fqbar)$ by acting on the coordinates $(y_1,\ldots ,y_d)$ of any affine chart of $y$ via
\[\Frob_q(y_1,\ldots ,y_d):=(y_1^q,\ldots ,y_d^q).\]
A point $y\in Y(\Fqbar)$ will be fixed by $\Frob_q$ precisely when $y\in Y(\F_q)$.  Thus
\[Y(\F_q)=\Fix (\Frob_q: Y(\Fqbar)\to Y(\Fqbar)).\]

Fix a prime $\ell$ not dividing $q$, and let $\Q_\ell$ denote the $\ell$-adic rationals.  Let
$H^*_{\et}(Y_{/\Fqbar};\Q_\ell)$ (resp.\ $H^*_{\et,c}(Y_{/\Fqbar};\Q_\ell)$) denote the \'{e}tale cohomology groups (resp.\ compactly supported \'{e}tale cohomology groups) of the base change $Y_{/\Fqbar}$ of $Y$ to $\Fqbar$ (see, e.g., \cite{De2,Mi}).    Denote by $\Qb_\ell(-i)$ the rank 1 $\Gal(\Fqbar/\Fb_q)$-representation on which Frobenius acts by $q^i$.

Let $\V$ be a constructible, rational $\ell$-adic sheaf on $Y$ (cf. e.g. \cite{FK}).  If $y\in Y(\Fqbar)$ is a fixed point for the action of $\Frob_q$, then $\Frob_q$ acts on the stalk $\V_y$ over $y$.  Attached to this action is its trace $\tr\big(\Frob_q:\V_y\to \V_y)$.  The {\em twisted Grothendieck--Lefschetz Trace Formula} (\cite[Theorem II.3.14]{FK} and \cite[6.1.1.1]{De2}) gives:

\begin{equation}
\label{eq:GLtwisted}
\sum_{y\in Y(\F_q)}\tr\big(\Frob_q:\V_y\to\V_y)
=\sum_{i=0}^{2\dim(Y)}(-1)^i\tr\big(\Frob_q: H^i_{\et,c}(Y;\V)\to H^i_{\et,c}(Y;\V)\big)
\end{equation}

When $Y$ is smooth, Poincar\'{e} duality for \'{e}tale cohomology \cite[Theorem 24.1]{Mi} gives
\begin{equation}
\label{eq:PD}
H^i_{\et,c}(Y_{/\Fqbar};\V)\iso H^{2\dim(Y)-i}_{\et}(Y_{/\Fqbar};\V(-\dim(Y)))^*.
\end{equation}
Plugging this into Equation \eqref{eq:GLtwisted} gives, for smooth $Y$:
\begin{equation}
\label{eq:smoothGLtwisted}
\sum_{y\in Y(\F_q)}\tr\big(\Frob_q:\V_y\to\V_y)
=q^{\dim (Y)} \sum_{i=0}^{2\dim(Y)}(-1)^i\tr\big(\Frob_q: H^i_{\et}(Y;\V)^*\to H^i_{\et}(Y;\V)^*\big)
\end{equation}

\paragraph{\boldmath$S_n$-schemes.}  Now let $Z$ be smooth and quasi-projective over $\Z[1/N]$.  Suppose that the symmetric group $S_n$ acts generically freely on $Z$ by automorphisms, and let $p:Z\to Y$ denote the quotient map.   By \cite[Theorem p. 63 and Remark p. 65 (Ch. 2.7)]{Mu}, $Y$ is a scheme. It is typically not smooth even when $Z$ is smooth.

Recall that any finite-dimensional representation of $S_n$ over a field of characteristic $0$ is defined over $\Q$.  There is a bijective correspondence between isomorphism classes of finite-dimensional $S_n$-representations and finite-dimensional constructible sheaves on $Y$ that become isomorphic to $\Q_\ell^{\oplus n}$ when pulled back to $Z$:  Given an $S_n$-representation $V$ over $\Q_\ell$, one can form an $S_n$-equivariant, locally constant sheaf $\V$ over $Z$ with fiber $V$. Pushing forward to $Y$ and taking $S_n$ invariants, i.e. $(p_\ast\V)^{S_n}$, we obtain a constructible sheaf of $\Q_\ell$ vector spaces over $Y$ which is a sheaf-theoretic analogue of the usual topological diagonal quotient ``$Z\times_{S_n}V$''.

Suppose that $y\in Y(\Fqbar)$ is fixed by $\Frob_q$.  Then $\Frob_q$ acts on the fiber $p^{-1}(y)$. Now $S_n$ acts transitively on $p^{-1}(y)$ with some stabilizer $H$ (not depending on $\tilde{y}\in p^{-1}(y)$), and so we can identify $p^{-1}(y)$ with $S_n/H$.  The $\Frob_q$ action on $p^{-1}(y)$ commutes with this $S_n$ action, and so it is determined by its action on a single basepoint, which we choose once and for all to be $H$.  Now $\Frob_q(H)=\sigma_y H$ for $\sigma_y\in S_n$.  Following Gadish \cite{Ga}, for any $S_n$-representation $V$ and any coset $\sigma H$ of $S_n$,  we set
\[\chi_V(\sigma H):=\frac{1}{|H|}\sum_{h\in H}\chi_V(\sigma h).\]
With this notation we have:
\begin{equation}
\label{eq:trace:frob}
\tr(\Frob_q:\V_y\to\V_y)=\chi_V(\sigma_y H)
\end{equation}
which we denote simply by $\chi_V(\Frob_q;\V_y)$.  More generally:

\begin{definition}
\label{definition:P(y)}
For any class function $P$, and any $y\in Y$ fixed by $\Frob_q$, define $P(y)$ by:
\begin{equation}\label{eq:charpoly}
    P(y):=\frac{1}{|H|}\sum_{h\in H}P(\sigma_y h).
\end{equation}
\end{definition}

An elementary check shows that the definitions above are independent of the choice of coset $H$, since the action of $S_n$ is transitive on fibers.

Plugging Equation \eqref{eq:trace:frob} into Equation \eqref{eq:GLtwisted} now gives:

\begin{equation}
\label{eq:GLtwisted2}
\sum_{y\in Y(\F_q)}\chi_V(\Frob_q;\V_y)
=\sum_{i=0}^{2\dim(Y)}(-1)^i\tr\big(\Frob_q: H^i_{\et,c}(Y;\V)\to H^i_{\et,c}(Y;\V)\big)
\end{equation}
The right-hand side of \eqref{eq:GLtwisted2} could be computed from
the eigenvalues $\lambda_{ij}$ of $\Frob_q$ on each $H^i_{\et,c}(Y;\Q_\ell)$.  Typically one only has estimates on $|\lambda_{ij}|$.  For example, for $Y$ smooth and proper, the Riemann Hypothesis for finite fields (proved by Deligne) gives that $|\lambda_{ij}|=q^{i/2}$. Many natural examples $Y$, including many of those we study in this paper, are not proper, and finding the $\lambda_{ij}$ is more difficult.

Given that we only have general bounds on the eigenvalues of $\Frob_q$, to bound the traces of $\Frob_q$ we must determine the dimensions of each $H^i_{\et,c}(Y;\V)$.  To do this, we follow the argument in \S 3.3 of \cite{CEF2}.  First note that the pullback $\widetilde{\V}$ of $\V$ to $Z$ is trivial.  We then compute:

\begin{equation}
\begin{array}{lll}
H^i_{\et,c}(Y;\V)&\cong H^i_{\et,c}(Z;\widetilde{\V})^{S_n}&\text{by transfer} \\
&&\\
&\cong (H^i_{\et,c}(Z;\Q_\ell)\otimes V)^{S_n}&\text{by triviality of $\widetilde{\V}|_Z$}\\
&&\\
&\cong (H^{2\dim(Z)-i}_{\et}(Z;\Q_\ell(\dim(Z)))^\ast\otimes V)^{S_n}&\text{by Poincar\`e duality}\\
&\\
&\cong H^{2\dim(Z)-i}_{\et}(Z;\Q_\ell(\dim(Z)))^\ast\otimes_{\Q_\ell[S_n]} V
\end{array}\label{eq:transfer1}
\end{equation}

Because every $S_n$-representation is self-dual, it follows that
\begin{equation}
\label{eq:dim:is:inner:prod}
    \dim_{\Q_\ell}H^i_{\et,c}(Y;\V)=\big\langle V,H_{\et}^{2\dim(Z)-i}(Z;\Q_\ell)\big\rangle_{S_n}
\end{equation}
where $\langle V,W\rangle_{S_n}$ is the usual inner product of $S_n$-representations $V$ and $W$: \[\langle V,W\rangle=\dim_{\Q_\ell}\Hom_{\Q_\ell[S_n]}(V,W).\]

\subsection{Co-FI schemes with convergent \'{e}tale cohomology}

Now that we have discussed schemes, and $S_n$-schemes, we are ready to discuss sequences of $S_n$-schemes.

Let $Z$ be a co-FI scheme, smooth and quasi-projective over $\Z[1/N]$.   For each $i\geq 0$, the \'etale cohomology $H^i_{\et}(Z_\bullet/_{\Fqbar};\Q_\ell)$ is an FI-module over $\Q_\ell$.  We want to consider the implications of finite generation of this FI-module for point-counting problems over $\F_q$ for the sequence of schemes $Z_n/S_n$ (cf. \cite[Theorem p. 63 and Remark p. 65]{Mu}).

As discussed in \cite{CEF1}, any partition $\lambda$ of any $k\geq 1$ determines a finitely-generated FI-module $V(\lambda)$ with $V(\lambda)_n$ being the irreducible representation of $S_n$ corresponding to the partition $(n-|\lambda|)+\lambda$.
\begin{definition}
    Let $K$ be a field, and let $M_\bullet$ be a finitely generated $\Gal(\Kbar/K)$-FI-module over $\Q_\ell$ with stable range $N$. Let $\lambda$ be a partition of $n$, let $V(\lambda)$ the associated $FI$-module, and let $D=\max\{N,\lambda_1\}$. Define the \emph{stable $\lambda$-isotypic part $M_\lambda$ of $M$}  to be the $\Gal(\Fqbar/\F_q)$-module
    \begin{equation*}
        M_\lambda:=M_D\otimes_{\Q_\ell[S_D]}V(\lambda)_D.
    \end{equation*}
    More generally, for a character polynomial $P$, we define the \emph{stable $P$-isotypic part of $M$} to be the $\Q_\ell$-virtual Galois module $M_P$ obtained as a linear combination of the $M_\lambda$, with the sum taken in the representation ring of $\Gal(\Kbar/K)$ with $\Q_\ell$-coefficients.
\end{definition}

Lemma \ref{lemma:stableisotypics} shows that for $n\ge D$, there are canonical Galois-equivariant isomorphisms
\begin{equation*}
    M_n\otimes_{\Q_\ell[S_n]}V(\lambda)_n\to^\cong M_{n+1}\otimes_{\Q_\ell[S_{n+1}]}V(\lambda)_{n+1}
\end{equation*}
and similarly for the stable $P$-isotypic parts for $n\ge D$.

We can now give the following theorem, which generalizes earlier special cases by Ellenberg-Venkatesh-Westerland \cite{EVW}, Ellenberg \cite{E}, and Church-Ellenberg-Farb \cite{CEF2}.  Its proof is along the exact same lines of the previous proofs.  We hope that the generality of the statement here will be useful in future work.

\begin{theorem}[{\bf Convergent Grothendieck-Lefschetz}]
\label{theorem:arithmetic:stats1}
Let $Z$ be a smooth, quasi-projective co-FI over $\F_q$, and set $Y_n:=Z_n/S_n$ (we do not assume $Y_n$ smooth over $\Z[1/N]$). Assume that for each $i\geq 0$ the FI-module $H^i_{\et}(Z_{n/\Fqbar};\Q_\ell)$ is finitely generated, and, for a character polynomial $P$, denote by $H^i_{\et}(Z)^\ast_P$ the dual of the stable $P$-isotypic part. If $Z$ has convergent \'etale cohomology over $\Fqbar$ , then for any character polynomial $P$ :
\begin{equation}
\label{eq:convlim1}
\lim_{n \to \infty} q^{-\dim Y_n} \sum_{y \in Y_n(\F_q)} P(y)=\sum_{i=0}^\infty (-1)^i\tr\left(\Frob_q\circlearrowleft H^i_{\et}(Z)_P^\ast\right),
\end{equation}
and, taking the absolute value :
\begin{equation}
\label{eq:convlim2}
\lim_{n \to \infty} q^{-\dim Y_n} |\sum_{y \in Y_n(\F_q)} P(y)| \leq \sum_{i=0}^\infty
\frac{\langle P,H_{\et}^i(Z)\rangle}{q^{i/2}}<\infty.
\end{equation}

If $Z$ only has weakly convergent cohomology with convergence rate $\lambda$, then \eqref{eq:convlim1} and \eqref{eq:convlim2} hold for all $q>\lambda$.
\label{th:limit}
\end{theorem}

\begin{remark}\mbox{}
    \begin{enumerate}
        \item Specializing the Theorem~\ref{th:limit} to the case $P=1$ gives
            \begin{equation}
                \lim_{n\to\infty}q^{-\dim Y_n} |Y_n(\F_q)|=\sum_{i=0}^\infty(-1)^i\tr\left(\Frob_q\circlearrowleft H^i_{\et}(Y)^\ast\right),
            \end{equation}
            where $H^i_{\et}(Y)^\ast$ denotes the stable rational \'etale cohomology of the sequence $Y_1,Y_2,\ldots$.
        \item The bound \eqref{eq:convlim2} is sharp, as is seen by taking $Z_n=(\mathbb{P}^1)^n$, $Y_n=\mathbb{P}^n$, and $P=1$.
     \end{enumerate}
\end{remark}

\begin{proof}[Proof of Theorem~\ref{th:limit}]
    Because all of the equations in the statement of the theorem are $\Q_\ell$-linear in $P$, it suffices to prove the theorem for $P=P_\lambda$, the character polynomial of the finitely generated FI-module $V(\lambda)$(cf.\ Theorem~\ref{theorem:FIproperties2} above).  Let $\V_n$ correspond to the twisted sheaf on $Y_n$ corresponding to the representation $V(\lambda)_n$.

    We show that the left side of \eqref{eq:convlim1} converges by showing that the sequences
    \begin{align}\label{eq:cauchy}
        n &\mapsto q^{-\dim Y_n} \sum_{y\in Y_n(\F_q)} P(y)
    \end{align}
    is Cauchy. To start, note that
    \begin{align}
        \sum_{y\in Y_n(\F_q)} P(y)&=\sum_{i=0}^{2\dim(Y_n)}(-1)^i\tr\left(\Frob_q\colon H^i_{\et,c}(Y_n;\V_n)\to H^i_{\et,c}(Y_n;\V_n)\right)\nonumber\\
    &=\sum_{i=0}^{2\dim(Z_n)}(-1)^i\tr\left(\Frob_q\circlearrowleft H^{2\dim(Z_n)-i}_{\et}(Z_n;\Q_\ell(\dim(Z_n)))^\ast\otimes_{\Q_\ell[S_n]} V(\lambda)_n)\right)\tag{by Equation \eqref{eq:transfer1}}\nonumber\\
    &=\sum_{i=0}^{2\dim(Z_n)}(-1)^i\tr\left(\Frob_q\circlearrowleft H^i_{\et}(Z_n;\Q_\ell(\dim(Z_n)))^\ast\otimes_{\Q_\ell[S_n]}V(\lambda)_n^\ast\right).\label{eq:pdupstairs}
    \end{align}
    where the last equation uses the self-duality of $S_n$-representations.

    Denote by $N(n,P)$ the slope of stability of $H^\ast_{\et}(Z_\bullet;\Q_\ell)$ for $V(\lambda)$, i.e. the number such that for all $i\le N(n,\lambda)$,
    \begin{equation*}
        H^i_{\et}(Z_n;\Q_\ell)\otimes_{\Q_\ell[S_n]}V(\lambda)_n\cong H^i_{\et}(Z)_P.
    \end{equation*}
    Let $F_P(i)$ denote the subexponential function in Definition~\ref{def:convergent} guaranteed by the assumption that $Z$ has convergent \'{e}tale cohomology. Then, for $n>m$:
    \begin{align*}
        |q^{-\dim Y_n}&(\sum_{y\in Y_n(\F_q)} P(y)) - q^{-\dim Y_m}(\sum_{y\in Y_m(\F_q)} P(y))|\\
        =&|\sum_{i=0}^{2\dim(Z_n)} (-1)^i q^{-\dim(Z_n)}\tr\left(\Frob_q\circlearrowleft H^i_{\et}(Z_n;\Q_\ell(\dim(Z_n)))^\ast\otimes_{\Q_\ell[S_n]}V(\lambda)_n^\ast\right)\\
        &-\sum_{i=0}^{2\dim(Z_m)}(-1)^iq^{-\dim(Z_m)}\tr\left(\Frob_q\circlearrowleft H^i_{\et}(Z_m;\Q_\ell(\dim(Z_m)))^\ast\otimes_{\Q_\ell[S_m]}V(\lambda)_m\right)|\tag{by Equations \eqref{eq:GLtwisted} and \eqref{eq:pdupstairs}}\\
        \le &\sum_{i=0}^{\infty}q^{-i/2}|\langle P,H^i_{\et}(Z_n;\Q_\ell)\rangle- \langle P,H^i_{\et}(Z_m;\Q_\ell)\rangle|\tag{by Deligne}\\
    =&\sum_{i=N(m,P)}^\infty q^{-i/2}|\langle P,H^i_{\et}(Z_n;\Q_\ell)\rangle- \langle P,H^i_{\et}(Z_m;\Q_\ell)\rangle|\tag{by \'etale representation stability}\\
    \le & \sum_{i=N(m,P)}^\infty 2 q^{-i/2} F_P(i) \tag{by convergent cohomology}.
\end{align*}
Because $N(m,P)$ tends to $\infty$ with $m$, and because $F_P(i)$ is sub-exponential in $i$, we see that the sequence \eqref{eq:cauchy} is Cauchy. Similarly, we see that the right side of \eqref{eq:convlim1}
\begin{equation*}
    \sum_{i=0}^\infty(-1)^i\tr(\Frob_q\circlearrowleft H^i_{\et}(Z)_P^\ast)
\end{equation*}
converges as a consequence of the existence of the stable $P$-isotypic part, Deligne's bounds on the eigenvalues of $\Frob_q$ and the existence of the sub-exponential bounds $F_P(i)$.

It remains to show that the two limits agree. For this, we have
\begin{align}\label{eq:pfconvlim1}
    |q^{-\dim Y_n} &\sum_{y\in Y_n(\F_q)} P(y)-\sum_{i=0}^{2\dim(Z_n)}(-1)^i\tr(\Frob_q\circlearrowleft H^i_{\et}(Z)_P^\ast)|\\
    =&|\sum_{i=0}^{2\dim(Z_n)}(-1)^i\big(\tr\left(\Frob_q\circlearrowleft H^i_{\et}(Z_n;\Q_\ell(\dim(Z_n)))^\ast\otimes_{\Q_\ell[S_n]}V(\lambda)_n^\ast\right)\nonumber\\
    &-\tr\left(\Frob_q\circlearrowleft H^i_{\et}(Z)_P^\ast\right)\big)|\tag{by Equation \eqref{eq:pdupstairs}}\nonumber\\
    =&|\sum_{i=N(n,P)+1}^{2\dim(Z_n)}(-1)^i\big(\tr\left(\Frob_q\circlearrowleft H^i_{\et}(Z_n;\Q_\ell(\dim(Z_n)))^\ast\otimes_{\Q_\ell[S_n]}V(\lambda)_n^\ast\right)\nonumber\\
    &-\tr\left(\Frob_q\circlearrowleft H^i_{\et}(Z)_P^\ast\right)\big)|\tag{by \'etale representation stability}\nonumber\\
    \le &\sum_{i=N(n,P)+1}^{2\dim(Z_n)} 2q^{-i/2}F_P(i)\tag{by Deligne and convergent cohomology}\nonumber.
\end{align}
Because $F_P(i)$ is subexponential in $i$, we conclude that \eqref{eq:pfconvlim1} becomes arbitrarily small as $n$ approaches $\infty$, which proves the theorem.
\end{proof}

We can now prove Theorem C from the introduction, as well as the following.
\begin{theorem}[{\bf Statistics for \boldmath$\Sym^nX$}]
\label{theorem:extra2}
Theorem C with $\PConf_n(Y)$ (resp.\ $\UConf_n(Y)$) replaced by $Y^n$ (resp.\ $\Sym^n(Y)$) holds.
\end{theorem}

\begin{proof} [Proof of Theorem C and Theorem~\ref{theorem:extra2}]
Theorem A gives that, for each $i\geq 0$, the FI-modules $H^*_{\et}(X^\bullet;\Q_\ell)$ and $H^*_{\et}(\PConf_\bullet(X);\Q_\ell)$ are finitely generated.
Corollary~\ref{corollary:bounding:products} (resp.\ Theorem~\ref{theorem:subexp2}) gives that the singular cohomology $H^*(X^\bullet;\Q)$ (resp.\ $H^*(\PConf_\bullet(X);\Q)$) of the co-FI scheme $X^\bullet$ (resp.\ $\PConf_\bullet(X)$) is convergent.  Since each of these co-FI schemes is normally compactifiable, it follows from base change (cf.\  Theorem~\ref{thm:basechange} above)  that the corresponding \'{e}tale cohomology groups are convergent.  Now apply Theorem~\ref{theorem:arithmetic:stats1}.
\end{proof}

In special cases it is possible to compute the right hand side of Equation \eqref{eq:convlim1} explicitly.
\begin{example}
    When $X=\Ab^r$, we can explicitly compute polynomial statistics on $\UConf_n(\Ab^r)$, extending the main theorem of \cite{CEF2}. Indeed, the computations of Arnol'd \cite{Ar} and F. Cohen \cite[\S 2]{Coh} combine with results of Bj\"orner--Ekedahl \cite[Theorem 4.9]{BE} to show that $H^\ast_{\et}(\PConf_n(\Ab^r)_{\Fqbar};\Q_\ell)$ is a graded algebra generated by classes in degree $2r-1$ with eigenvalues of $\Frob_q$ equal to $q^{r}$. As a result, for any character polynomial $P_\lambda$:
    \begin{align*}
        \tr(\Frob_q\colon &H^i_{\et}(\UConf_n(\Ab^r);\V)^\ast\to H^i_{\et}(\UConf_n(\Ab^r);\V)^\ast)\\
        &\\
        &=\left\{\begin{array}{ll}
        0 & \text{if } i\neq k(2r-1)\\
        q^{-kr}\langle P_\lambda,H^i_{\et}(\PConf_n(\Ab^r);\Q_\ell)\rangle & \text{if } i=k(2r-1)
        \end{array}\right.
    \end{align*}
    Here we have, as above, applied Poincar\`e duality to replace the compactly supported cohomology of the smooth schemes $\UConf_n(\Ab^r)$ with (the Tate twist of) the dual of ordinary \'etale cohomology.  We thus have, for all $P$,
    \begin{equation*}
        \lim_{n \to \infty} q^{-nr} \sum_{y \in\UConf_n(\Ab^r)(\F_q)} P(y) = \sum_{i=0}^\infty
        (-1)^{i(2r-1)} q^{-ir}\langle P,H_{\et}^{i(2r-1)}(\PConf(\Ab^r))\rangle.
    \end{equation*}
\end{example}

\bigskip{\noindent
Dept. of Mathematics, University of Chicago\\
E-mail: farb@math.uchicago.edu, wolfson@math.uchicago.edu
}

\end{document}